\documentclass{amsart}
\usepackage[margin=1.5in]{geometry}
\usepackage{mathabx}
\usepackage{amsfonts}
\usepackage{lineno,hyperref}
\usepackage{amsmath}
\usepackage{amssymb}
\usepackage{amsthm}
\usepackage{cancel}
\usepackage{tikz-cd}
\usepackage{arydshln}
\usepackage{graphicx}

\newcommand{\RR}{\mathbb{R}}
\newcommand{\CC}{\mathbb{C}}

\newcommand{\ZZ}{\mathbb{Z}}
\newcommand{\QQ}{\mathbb{Q}}
\newcommand{\OO}{\mathcal{O}}

\newcommand{\nrm}{\text{nrm}}
\newcommand{\tr}{\text{tr}}
\newcommand{\disc}{\text{disc}}

\newcommand{\Mat}{\text{Mat}}

\makeatletter
\newtheorem*{rep@theorem}{\rep@title}
\newcommand{\newreptheorem}[2]{%
\newenvironment{rep#1}[1]{%
 \def\rep@title{#2 \ref{##1}}%
 \begin{rep@theorem}}%
 {\end{rep@theorem}}}
\makeatother

\theoremstyle{plain}
\newtheorem{definition}{Definition}[section]
\newtheorem{theorem}{Theorem}[section]
\newreptheorem{theorem}{Theorem}
\newtheorem{corollary}{Corollary}[section]
\newtheorem{lemma}{Lemma}[section]

\theoremstyle{remark}
\newtheorem{example}{Example}
\newtheorem{remark}{Remark}[section]

\begin{document}
\title{Algebraic Groups Constructed From Orders of Quaternion Algebras}
\author{Arseniy (Senia) Sheydvasser}
\address{Department of Mathematics, Graduate Center at CUNY, 365 5th Ave, New York, NY 10016}
\email{ssheydvasser@gc.cuny.edu}


\date{\today}


\begin{abstract}
We show that all spin groups of non-definite, quinary quadratic forms over a field with characteristic $0$ can be represented as $2\times 2$ matrices with entries in an associated quaternion algebra. Over local and global fields, we further study maximal arithmetic subgroups of such groups, and show that examples can be produced by studying orders of the quaternion algebra. In both cases, we relate the algebraic properties of the underlying rings to sufficient and necessary conditions for the groups to be isomorphic and/or conjugate to one another.
\end{abstract}

\maketitle

\section{Introduction:}
The idea of representing elements of $\text{M\"{o}b}(\RR^n)$ with $2\times 2$ matrices with entries in a Clifford algebra goes back at least to Vahlen \cite{Vahlen1902}, and was later popularized by Ahlfors \cite{Ahlfors1986}. More recently, this approach was used by the author to construct explicit examples of integral, crystallographic sphere packings \cite{Sheydvasser2019}; briefly, these are generalizations of the classical Apollonian gasket which arise from hyperbolic lattices. Such packings were formally defined by Kontorovich and Nakamura \cite{KontorovichNakamura2019}, although they were studied in various forms previously. How to define $\text{M\"{o}b}(\RR)$ and $\text{M\"{o}b}(\RR^2)$ in terms of the real and complex numbers is well-known. In order to describe $\text{M\"{o}b}(\RR^3)$ in terms of $2 \times 2$ matrices, we proceed as follows: let $H_\mathbb{R}$ be the standard Hamilton quaternions, and define an involution $(w + xi + yj + zk)^\ddagger = w + xi + yj - zk$. One can then define the set
	\begin{align*}
	SL^\ddagger(2,H_\mathbb{R}) = \left\{\begin{pmatrix} a & b \\ c & d \end{pmatrix} \in \Mat(2,H_\mathbb{R}) \middle| ab^\ddagger = ba^\ddagger, \ cd^\ddagger = dc^\ddagger, \ ad^\ddagger - bc^\ddagger = 1\right\}.
	\end{align*}
	
\noindent One checks that this is a group, and that $SL^\ddagger(2,H_\mathbb{R})/\{\pm I\} \cong \text{ M\"{o}b}(\RR^3)$---or, if one prefers, since $\text{ M\"{o}b}(\RR^3) \cong SO^+(4,1)$, $SL^\ddagger(2,H_\mathbb{R}) \cong \text{Spin}(4,1)$. However, one observes that there is nothing in the definition of $SL^\ddagger(2,H_\mathbb{R})$ that is specific to the real numbers: one can just as well choose any field $F$---we shall assume throughout that $\text{char}(F) \neq 2$ for convenience---any quaternion algebra $H$ over $F$, any orthogonal involution $\ddagger$ of $H$, and this will allow you to define a group $SL^\ddagger(2,H)$. It is evident that this a linear algebraic group; we shall prove that in fact it will be the spin group of a quadratic form over $F$, just as in the classical case $F = \mathbb{R}$. We shall show that the quaternion algebra $H$ determines the spin group up to isomorphism, and the quaternion algebra $H$ together with the involution $\ddagger$ determine the conjugacy class inside of an algebraic extension of $F$. Since $\RR$ contains the square roots of the norms of all elements in $H_\RR$, this is a consideration that doesn't come up in the Lie group case.

One of the benefits of representing spin groups in terms of quaternion algebras becomes evident in the special case where $F = K$ is a number field. In this case, it is known that if you take a quadratic form $q$ over $K$ and consider the linear algebraic group $G = SO(q)$, then every arithmetic subgroup of $G$ will be commensurable to $SO(q;\mathfrak{o}_K)$ where $\mathfrak{o}_K$ is the ring of integers of $K$. However, $SO(q;\mathfrak{o}_K)$ will not in general be a maximal arithmetic subgroup. On the other hand, we show that if $H$ is a rational quaternion algebra, then it is easy to find a maximal arithmetic subgroup of $SL^\ddagger(2,H)$---it suffices to find an order $\OO$ of $H$ that is closed under the involution $\ddagger$ and which is not contained in any larger order having this property. In that case, we shall demonstrate that $SL^\ddagger(2,\OO)$ is a maximal arithmetic subgroup of $SL^\ddagger(2,H)$; this can be seen as analogous to the statement that $SL(2,\ZZ)$ and the Bianchi groups are maximal arithmetic groups. There are known algorithms for finding maximal orders in quaternion algebras over number fields---for example, Ivanos and Ronyai \cite{IvanyosRonyai} gave an algorithm that works over general semisimple algebras over $\QQ$, and Voight \cite{VoightArticle} gave an efficient algorithm specifically for quaternion algebras. These can be adapted to give an efficient algorithm for constructing maximal $\ddagger$-orders over an arbitrary number field, and so our approach certainly gives an efficient means of computing maximal arithmetic subgroups of $SL^\ddagger(2,H)$ for rational quaternion algebras---we conjecture that this is true more generally for number fields.

\section{Summary of Main Results:}

In section \ref{Preliminaries}, we give a brief review of the basic definitions of orthogonal involutions and maximal $\ddagger$-orders. Our first main result occurs in section \ref{Is a Spin Group Section}, where we prove that $SL^\ddagger(2,H)$ is a spin goup. Specifically, there is a quadratic form $q_H$ such that the following is true.

\begin{reptheorem}{Simply Connectedness}
Let $H$ be a quaternion algebra over a field $F$ of characteristic not $2$, with orthogonal involution $\ddagger$. Then, as an algebraic group, $SL^\ddagger(2,H) \cong \text{Spin}(q_H)$. In particular, it is an absolutely almost simple, simply connected group.
\end{reptheorem}

\noindent However, one can say much more. In Section \ref{Correspondence Isomorphism Section}, we prove that all spin groups of indefinite, quinary quadratic forms arise as groups $SL^\ddagger(2,H)$.

\begin{reptheorem}{Correspondence between QAs and Spin Groups}
Let $F$ be a characteristic $0$ field. There there is a bijection
	\begin{align*}
	\left\{\substack{\text{Isomorphism classes of} \\ \text{quaternion algebras over } F}\right\} &\rightarrow \left\{\substack{\text{Isomorphism classes of} \\ \text{spin groups of indefinite,} \\ \text{quinary quadratic forms over } F} \right\} \\
	[H] & \mapsto \left[SL^\ddagger(2,H)\right].
	\end{align*}
\end{reptheorem}

\begin{remark}
The restriction to characteristic $0$---rather than simply not $2$---appears due to the proof using the classification of Lie algebras and the separability of extensions of $F$. It may be possible to remove this restriction, but as our chief interest shall be in number fields, proving the result for characteristic $0$ fields will be more than sufficient.
\end{remark}

\noindent We can get more fine information than simply determining whether these groups are isomorphic. Specifically, whether or not two groups $SL^{\ddagger_1}(2,H_1)$ and $SL^{\ddagger_2}(2,H_2)$ are conjugate inside of some larger spin group can be determined by studying what algebras with involution both $(H_1,\ddagger_1)$ and $(H_2, \ddagger_2)$ embed in.

\begin{reptheorem}{Isomorphisms and Embeddings}
Let $F$ be a characteristic $0$ field, and let $H_1$, $H_2$ be isomorphic quaternion algebras over $F$, with orthogonal involutions $\ddagger_1$ and $\ddagger_2$. Then there exists a quaternion algebra $H$ over a field $F' \supset F$ with an orthogonal involution $\ddagger$ such that $(H_1, \ddagger_1)$ and $(H_2, \ddagger_2)$ embed inside of $(H, \ddagger)$. Furthermore, for any $F'$ and $(H,\ddagger)$ satisfying this condition, the natural embeddings of $SL^{\ddagger_1}(2,H_1)$ and $SL^{\ddagger_2}(2,H_2)$ are conjugate inside of $SL^\ddagger(2,H)$.
\end{reptheorem}

\noindent All of the above theorems only require the quaternion algebra $H$ to be considered up to ring isomorphism, which does not necessarily preserve the involution $\ddagger$. A more refined notion of isomorphism can be produced by asking under what circumstances spin groups $SL^\ddagger(2,H)$ are related by conjugation that preserves the subgroup $SL(2,F)$.

\begin{reptheorem}{Algebraic Group Isomorphism Fixing Base Group}
Let $F$ be a characteristic $0$ field, and let $H_1$, $H_2$ be isomorphic quaternion algebras over $F$, with orthogonal involutions $\ddagger_1$ and $\ddagger_2$. The following two statements are equivalent.
	\begin{enumerate}
		\item If $H$ is a quaternion algebra over a field extension $F'$ with orthogonal involution $\ddagger$ such that $(H_1, \ddagger_1)$, $(H_2, \ddagger_2) \hookrightarrow (H,\ddagger)$, then there exists an element $\gamma \in SL^\ddagger(2,H)$ such that $\gamma SL^{\ddagger_1}(2,H_1) \gamma^{-1} = SL^{\ddagger_2}(2,H_2)$, and $\gamma SL(2,F) \gamma^{-1} = SL(2,F)$.
		\item $(H_1, \ddagger_1) \cong (H_2, \ddagger_2)$.
	\end{enumerate}
\end{reptheorem}

\noindent Section \ref{Maximal Subgroups Section} introduces the subgroups $SL^\ddagger(2,\OO)$, and after proving some basic results about invariants of such groups, shows that they are maximal arithmetic groups.

\begin{reptheorem}{Maximal Arithmetic Groups}
Let $H$ be a rational quaternion algebra with orthogonal involution $\ddagger$. Let $\OO$ be a maximal $\ddagger$-order of $H$. Then $SL^\ddagger(2,\OO)$ is a maximal arithmetic subgroup of $SL^\ddagger(2,H)$.
\end{reptheorem}

\noindent In Section \ref{Isomorphisms Between Arithmetic Groups}, we produce some partial results regarding necessary and sufficient conditions for different groups of the form $SL^\ddagger(2,\OO)$ to be isomorphic to each other. As an initial necessary condition, we show that if $SL^{\ddagger_1}(2,\OO_1)$ is isomorphic to $SL^{\ddagger_2}(2,\OO_2)$, then $\OO_1$ and $\OO_2$ must be matrix isomorphic.

\begin{reptheorem}{Isomorphism Necessary Condition}
Let $H_1,H_2$ be quaternion algebras over a number field $K$, with orthogonal involutions $\ddagger_1$ and $\ddagger_2$. Let $\OO_1,\OO_2$ be maximal $\ddagger_1,\ddagger_2$-orders of $H_1$ and $H_2$, respectively. If the group $SL^{\ddagger_1}(2,\OO_1)$ is isomorphic to $SL^{\ddagger_2}(2,\OO_2)$, then $H_1 \cong H_2$ and $\Mat(2,\OO_1) \cong \Mat(2, \OO_2)$. In particular, $\disc(\OO_1) = \disc(\OO_2)$.
\end{reptheorem}

\noindent We further show that this result seems to be in some sense sharp---specifically, we produce an example of orders $\OO_1$ and $\OO_2$ such that $SL^{\ddagger_1}(2,\OO_1) \cong SL^{\ddagger_2}(2,\OO_2)$, but $\OO_1 \ncong \OO_2$. This unpleasant situation is not mirrored by the more restrictive notion of isomorphism introduced in Section \ref{Correspondence Isomorphism Section}; there, as with the algebraic group case, we see that the notion of isomorphism preserving involution exactly corresponds to conjugation preserving $SL(2,F)$.

\begin{reptheorem}{Special Isomorphism of Orders Is Equivalent to Special Isomorphism of Groups}
Let $H$ a quaternion algebra over a number field $K$ with orthogonal involution $\ddagger$. Let $\OO_1, \OO_2$ be maximal $\ddagger$-orders of $H$. Let $K'$ be a field extension of $K$ containing the splitting fields of $X^2 = \nrm(x)$ for all $x \in H$. Then the following are equivalent.
	\begin{enumerate}
	\item There exists $\gamma \in SL^\ddagger(2,H \otimes_K K')$ such that $\gamma SL^\ddagger(2,\OO_1) \gamma^{-1} = SL^\ddagger(2,\OO_2)$ and $\gamma SL(2,K) \gamma^{-1} = SL(2,K)$.
	\item $(\OO_1,\ddagger) \cong (\OO_2,\ddagger)$.
	\end{enumerate}
\end{reptheorem}

\noindent Finally, in Section \ref{Conjugacy Section}, we examine under what circumstances subgroups of the form $SL^\ddagger(2,\OO)$ are conjugate inside of the algebraic group in which they sit. Our main result is that this problem is actually local---that is, if the groups are conjugate over all of the localizations of the number field, then they are conjugate over the number field itself.

\begin{reptheorem}{Local Rigidity}
Let $H$ be a quaternion algebra over a number field $K$, with orthogonal involution	$\ddagger$. Let $\OO_1, \OO_2$ be maximal $\ddagger$-orders of $H$. Then $SL^\ddagger(2,\OO_1)$ is conjugate to $SL^\ddagger(2,\OO_2)$ in $SL^\ddagger(2,H)$ if and only if for every prime ideal $\mathfrak{p}$ of $\mathfrak{o}_K$, $SL^\ddagger(2,\OO_{1,\mathfrak{p}})$ is conjugate to $SL^\ddagger(2,\OO_{2,\mathfrak{p}})$ in $SL^\ddagger(2,H_\mathfrak{p})$.
\end{reptheorem}

\subsection*{Acknowledgements:}

The author would like to thank Joseph Quinn for a very productive conversation about invariants and isomorphisms of hyperbolic quotient manifolds, which inspired many of the approaches used in this paper.

\section{Preliminaries:}\label{Preliminaries}

Throughout this paper, we use the following standard conventions.

\

\begin{tabular}{ll}
	$\mathfrak{o}_F$ & the ring of integers of $F$, if $F$ is a local or global field. \\
	$H$ & a quaternion algebra over $F$. \\
	$x \mapsto \overline{x}$ & the standard involution on $H$. \\ 
	$\nrm(x) = x\overline{x}$ & the (reduced) norm map on $H$. \\
	$\tr(x) = x + \overline{x}$ & the (reduced) trace map on $H$.\\
	$H^0$ & the trace $0$ subspace of $H$. \\
	$H^\times$ & the unit group of $H$. \\
  $\left(\frac{a,b}{F}\right)$ & the quaternion algebra generated by $i,j$, where $i^2 = a$, $j^2 = b$, $ij = -ji$.
\end{tabular}

\

\noindent We shall also make use of the terminology from \cite{Sheydvasser2017}, wherein the author defined the notion of a maximal $\ddagger$-order. Briefly, for any quaternion algebra $H$, an involution of the first type is an $F$-linear map $\varphi: H \rightarrow H$ such that
	\begin{enumerate}
		\item $\varphi(xy) = \varphi(y)\varphi(x)$ and
		\item $\varphi(\varphi(x)) = x$.
	\end{enumerate}
	
\noindent The best-known involution of the first type is the standard involution of quaternion conjugation. Every other such involution is related to each other by conjugation, and are known as orthogonal involutions. We shall always denote such an involution by a superscript of $\ddagger$. The most common example that we shall use is
	\begin{align*}
	(w + xi + yj + zij)^\ddagger = w + xi + yj - zij.
	\end{align*}
	
\noindent Any such involution will act as the identity on a subspace of dimension $3$, which we shall denote by $H^+$, and act as multiplication by $-1$ on a subspace of dimension $1$, which we shall denote by $H^-$. A homomorphism of algebras with involution $f: (A, \ddagger_1) \rightarrow (B, \ddagger_2)$ is a ring homomorphism with the property that $f(x^{\ddagger_1}) = f(x)^{\ddagger_2}$. If the two algebras with involution are isomorphic, we shall write this as $(A, \ddagger_1) \cong (B, \ddagger_2)$. If $H$ is a quaternion algebra with an orthogonal involution $\ddagger$, then we may consider orders of $H$ that are closed under $\ddagger$---we call these $\ddagger$-orders. A maximal $\ddagger$-order is one that is not contained in any strictly larger $\ddagger$-order.

This notion of maximal $\ddagger$-orders was previously studied by Scharlau \cite{Scharlau1974} over central simple algebras and under a slightly different name. The primary new contribution of the author was a full classification of maximal $\ddagger$-orders over local fields of characteristic not $2$. This allowed an easy criterion to check whether a $\ddagger$-order is maximal.

\begin{theorem}\cite[Theorem 1.1]{Sheydvasser2017}\label{MainTheoremOrderPaper}
Given a quaternion algebra $H$ over a local or global field $F$ of characteristic not $2$ and with an orthogonal involution $\ddagger$, the maximal $\ddagger$-orders of $H$ correspond to Eichler orders of the form $\OO \cap \OO^\ddagger$ with discriminant
	\begin{align*}
	\disc(H) \cap \iota(\disc(\ddagger)).
	\end{align*}
\end{theorem}

\noindent Here $\disc(H)$ is the product of the ramifying primes of $H$, $\disc(\ddagger) = x^2 \left(F^\times\right)^2$ for any non-zero $x \in H^-$, and
	\begin{align*}
	\iota: F^\times/\left(F^\times\right)^2 &\rightarrow \left\{\text{square-free ideals of } \mathfrak{o}_F\right\} \\
	[\lambda] &\mapsto \bigcup_{\lambda \in [\lambda] \cap \mathfrak{o}} \lambda \mathfrak{o}_F.
	\end{align*}
	
\noindent Given that the discriminants match, one may be tempted to ask whether all maximal $\ddagger$-orders are ring isomorphic. This is not so.

\begin{example}\label{Non Isomorphic Orders}
Define $(w + xi + yj + zij)^\ddagger = w + xi - yj + zij$. Then
	\begin{align*}
	\OO_1 &= \ZZ \oplus \ZZ i \oplus \ZZ \frac{1 + j}{2} \oplus \ZZ \frac{i + ij}{2} \subset \left(\frac{-1,-23}{\QQ}\right) \\
	\OO_2 &= \ZZ \oplus 3\ZZ i \oplus \ZZ \frac{1 + j}{2} \oplus \ZZ \frac{11 i + ij}{6} \subset \left(\frac{-1,-23}{\QQ}\right)
	\end{align*}
	
\noindent are both maximal $\ddagger$-orders by Theorem \ref{MainTheoremOrderPaper}, and in fact all of the localizations of $\OO_1$ and $\OO_2$ are isomorphic as algebras with involution. However, $\OO_1 \ncong \OO_2$, since $\OO_1^\times = \{1,-1,i,-i\}$, whereas $\OO_2^\times = \{1,-1\}$.
\end{example}

\section{$SL^\ddagger(2,H)$ is a Spin Group:}\label{Is a Spin Group Section}
With the conventions of Section \ref{Preliminaries}, we can define the set $SL^\ddagger(2,H)$ as follows.

\begin{definition}
Let $F$ be a field of characteristic not $2$. Let $H$ be a quaternion algebra over $F$, with orthogonal involution $\ddagger$. Then
	\begin{align*}
	SL^\ddagger(2,H) = \left\{\begin{pmatrix} a & b \\ c & d \end{pmatrix} \in \Mat(2,H) \middle| ab^\ddagger \in H^+, \ cd^\ddagger \in H^+, \ ad^\ddagger - bc^\ddagger = 1\right\}.
	\end{align*}
\end{definition}

\noindent That this is a group under standard matrix multiplication can either be checked by explicit computation, or by noting that
	\begin{align*}
	SL^\ddagger(2,H) = \left\{\gamma \in GL(2,H) \middle| \gamma \begin{pmatrix} 0 & ij \\ -ij & 0 \end{pmatrix} \overline{\gamma}^T = \begin{pmatrix} 0 & ij \\ -ij & 0 \end{pmatrix}\right\}.
	\end{align*}
	
\noindent In any case, one checks that inverses in this group are given by
	\begin{align*}
	\begin{pmatrix} a & b \\ c & d \end{pmatrix}^{-1} = \begin{pmatrix} d^\ddagger & -b^\ddagger \\ -c^\ddagger & a^\ddagger \end{pmatrix},
	\end{align*}
	
\noindent which will be important later. In the meantime, we start by proving that $SL^\ddagger(2,H)$ is a connected algebraic group.

\begin{lemma}\label{Proof of Connectedness}
If $F$ is a field of characteristic not equal to $2$, and $H$ is a quaternion algebra over $F$ with orthogonal involution $\ddagger$, then $SL^\ddagger(2,H)$ is a connected algebraic group.
\end{lemma}

\begin{proof}
Note that for any element of $SL^\ddagger(2,H)$, if $\nrm(d) \neq 0$, then
	\begin{align*}
	\begin{pmatrix} a & b \\ c & d \end{pmatrix} &= \begin{pmatrix} 1 & bd^{-1} \\ 0 & 1 \end{pmatrix} \begin{pmatrix} a - bd^{-1}c & 0 \\ c & d \end{pmatrix} \\
	&= \begin{pmatrix} 1 & bd^{-1} \\ 0 & 1 \end{pmatrix} \begin{pmatrix} a - bd^{-1}c & 0 \\ 0 & d \end{pmatrix} \begin{pmatrix} 1 & 0 \\ d^{-1}c & 1 \end{pmatrix}.
	\end{align*}
	
\noindent Furthermore, any element $u \in H^\times$ can be written in the form $u = 1 + xy$ for some $x,y \in H^+$, and so
	\begin{align*}
	\begin{pmatrix} 1 & 0 \\ -y(1 + xy)^{-1} & 1 \end{pmatrix} &\begin{pmatrix} 1 & x \\ 0 & 1 \end{pmatrix} \begin{pmatrix} 1 & 0 \\ y & 1 \end{pmatrix} \begin{pmatrix} 1 & -(1 + xy)^{-1}x \\ 0 & 1 \end{pmatrix} \\
	&= \begin{pmatrix} 1 & 0 \\ -y(1 + xy)^{-1} & 1 \end{pmatrix} \begin{pmatrix} 1 + xy & x \\ y & 1 \end{pmatrix} \begin{pmatrix} 1 & -(1 + xy)^{-1}x \\ 0 & 1 \end{pmatrix} \\
	&= \begin{pmatrix} 1 + xy & x \\ 0 & 1 - y(1 + xy)^{-1}x \end{pmatrix} \begin{pmatrix} 1 & -(1 + xy)^{-1}x \\ 0 & 1 \end{pmatrix} \\
	&= \begin{pmatrix} 1 + xy & 0 \\ 0 & 1 - y(1 + xy)^{-1}x \end{pmatrix} \\
	&= \begin{pmatrix} u & 0 \\ 0 & \left(u^{-1}\right)^\ddagger \end{pmatrix}.
	\end{align*}

\noindent Therefore, any element of $SL^\ddagger(2,H)$ with $\nrm(d) \neq 0$ can be written as a product of elements in
	\begin{align*}
	U^\ddagger(2,H) &= \left\{\begin{pmatrix} 1 & z \\ 0 & 1 \end{pmatrix} \middle| z \in H^+\right\} \\
	L^\ddagger(2,H) &= \left\{\begin{pmatrix} 1 & 0 \\ z & 1 \end{pmatrix} \middle| z \in H^+\right\}.
	\end{align*}

\noindent However, the collection $\mathcal{U}$ of matrices with $\nrm(d) \neq 0$ is an open, dense subset of $SL^\ddagger(2,H)$, hence $\mathcal{U} \cdot \mathcal{U} = SL^\ddagger(2,H)$, so in fact every element of $SL^\ddagger(2,H)$ can be written as a product of elements in $U^\ddagger(2,H)$ and $L^\ddagger(2,H)$. Clearly, $U^\ddagger(2,H) \cong L^\ddagger(2,H) \cong F \oplus F \oplus F$, so both of these groups are connected. But since $SL^\ddagger(2,H)$ is generated by these two subgroups, it is itself connected.
\end{proof}

This lemma has an immediate consequence---we can obtain an analog of the classic result that $SL(2,\CC)$ has a homomorphism into $SO^+(3,1)$. In our case, it is instead a homomorphism of $SL^\ddagger(2,H)$ into the connected component $SO^0$ of a special orthogonal group.

\begin{theorem}\label{Algebraic Group Exact Sequence}
Let $H$ be a quaternion algebra over a field $F$ not characteristic $2$, with orthogonal involution $\ddagger$. Define a quadratic form $q_H$ on $F^2 \oplus H^+$ by
	\begin{align*}
	q_H(s,t,z) = st - \nrm(z).
	\end{align*}
	
\noindent Then there is an exact sequence of algebraic groups
	\begin{align*}
	1 \rightarrow \left\{\pm I\right\} \rightarrow SL^\ddagger(2,H) \rightarrow SO^0(q_H) \rightarrow 1.
	\end{align*}
\end{theorem}

\begin{remark}
The special case where $K = \QQ$ and $H$ is positive definite was worked out in \cite{Sheydvasser2019}. We follow mostly the same argument.
\end{remark}

\begin{proof}[Proof of Theorem \ref{Algebraic Group Exact Sequence}]
Define a set
	\begin{align*}
	\mathcal{M}_H = \left\{M = \begin{pmatrix} a & b \\ c & d \end{pmatrix} \in \Mat(2,H)\middle|\overline{M}^T = M, \ ab^\ddagger, cd^\ddagger \in H^+\right\}.
	\end{align*}
	
\noindent It is easy to see that there is a bijection
	\begin{align*}
	F^2 \oplus H^+ &\rightarrow \mathcal{M}_H \\
	(s, t, z) &\mapsto \begin{pmatrix} s & z \\ \overline{z} & t \end{pmatrix},
	\end{align*}
	
\noindent taking the quadratic norm to the quasi-determinant---thus, we can identify these two sets. On the other hand, if $\gamma \in SL^\ddagger(2,H)$ and $M \in \mathcal{M}_H$, then it is easy to check that $\gamma M \overline{\gamma}^T \in \mathcal{M}_H$ as well, and has the same quasi-determinant as $\gamma$. Therefore, we have defined a morphism of algebraic groups $SL^\ddagger(2,H) \rightarrow O(q_H)$. Checking the action of the generators of $SL^\ddagger(2,H)$, one finds that the image is in fact inside $SO(q_H)$. For any element of the kernel,
	\begin{align*}
	\begin{pmatrix} 1 & 0 \\ 0 & 0 \end{pmatrix} &= \begin{pmatrix} a & b \\ c & d \end{pmatrix}\begin{pmatrix} 1 & 0 \\ 0 & 0 \end{pmatrix}\begin{pmatrix} \overline{a} & \overline{c} \\ \overline{b} & \overline{d} \end{pmatrix} \\
	&= \begin{pmatrix} a & 0 \\ c & 0 \end{pmatrix}\begin{pmatrix} \overline{a} & \overline{c} \\ \overline{b} & \overline{d} \end{pmatrix} \\
	&= \begin{pmatrix} \nrm(a) & a\overline{c} \\ c\overline{a} & \nrm(c) \end{pmatrix},
	\end{align*}

\noindent from which we conclude that $c = 0$ and $\nrm(a) = 1$. Similarly, the relation
	\begin{align*}
	\begin{pmatrix} 0 & 0 \\ 0 & 1 \end{pmatrix} &= \begin{pmatrix} a & b \\ 0 & d \end{pmatrix}\begin{pmatrix} 0 & 0 \\ 0 & 1 \end{pmatrix}\begin{pmatrix} \overline{a} & 0 \\ \overline{b} & \overline{d} \end{pmatrix} \\
	&= \begin{pmatrix} 0 & b \\ 0 & d \end{pmatrix}\begin{pmatrix} \overline{a} & 0 \\ \overline{b} & \overline{d} \end{pmatrix} \\
	&= \begin{pmatrix} \nrm(b) & b\overline{d} \\ d\overline{b} & \nrm(d) \end{pmatrix}
	\end{align*}
	
\noindent gives us that $b = 0$ and $\nrm(d) = 1$. Finally, we note that
	\begin{align*}
	\begin{pmatrix} 0 & z \\ \overline{z} & 0 \end{pmatrix} &= \begin{pmatrix} a & 0 \\ 0 & d \end{pmatrix}\begin{pmatrix} 0 & z \\ \overline{z} & 0 \end{pmatrix}\begin{pmatrix} \overline{a} & 0 \\ 0 & \overline{d} \end{pmatrix} \\
	&= \begin{pmatrix} 0 & az \\ d\overline{z} & 0 \end{pmatrix}\begin{pmatrix} \overline{a} & 0 \\ 0 & \overline{d} \end{pmatrix} \\
	&= \begin{pmatrix} 0 & az\overline{d} \\ d\overline{z}\,\overline{a} & 0 \end{pmatrix}
	\end{align*}
	
\noindent implies $az\overline{d} = z$ for all $z \in H^+$. Since $\nrm(d) = 1$, this is just to say that $az = zd$ for all $z \in H^+$, and since $ad^\ddagger = 1$, this is the same as saying that $az = z\overline{a^\ddagger}$ for all $z \in H^+$. It is easy to check this equation is satisfied only if $a \in F$, but since $\nrm(a) = 1$, we see that $a^2 = 1$, and therefore the kernel actually just consists of $\pm I$, as claimed. Since the kernel has dimension $0$, the dimension of the image is $\dim\left(SL^\ddagger(2,H)\right) = 10$, which is the same as the dimension of $SO(q_H)$. Furthermore, since $SL^\ddagger(2,H)$ is connected by Lemma \ref{Proof of Connectedness}, we conclude that the image of $SL^\ddagger(2,H)$ is the identity component of $SO(q_H)$.
\end{proof}

As a corollary, we get the desired result that $SL^\ddagger(2,H)$ is a spin group.

\begin{theorem}\label{Simply Connectedness}
Let $H$ be a quaternion algebra over a field $F$ not characteristic $2$, with orthogonal involution $\ddagger$. Then, as an algebraic group, $SL^\ddagger(2,H) \cong \text{Spin}(q_H)$. In particular, it is an absolutely almost simple, simply connected group.
\end{theorem}

In the special case where $F$ is a global field, we can use the strong approximation theorem proved by Knesser and Platonov \cite{Kneser1965, Platonov1969} to get an immediate but very useful corollary of Theorem \ref{Simply Connectedness}.

\begin{corollary}\label{Strong Approximation}
Let $H$ be a quaternion algebra over a global field $F$ not characteristic $2$, with orthogonal involution $\ddagger$. Let $S$ be the set of infinite places of $F$. Then $SL^\ddagger(2,H)$ has strong approximation with respect to $S$.
\end{corollary}

\section{Correspondence between Quaternion Algebras and Spin Groups:}\label{Correspondence Isomorphism Section}

Having established that all groups $SL^\ddagger(2,H)$ are spin groups, our next goal is to determine when two such groups are isomorphic, and which spin groups can be represented in such a manner. Our main result is the following.

\begin{theorem}\label{Correspondence between QAs and Spin Groups}
Let $F$ be a characteristic $0$ field. There there is a bijection
	\begin{align*}
	\left\{\substack{\text{Isomorphism classes of} \\ \text{quaternion algebras over } F}\right\} &\rightarrow \left\{\substack{\text{Isomorphism classes of} \\ \text{spin groups of indefinite,} \\ \text{quinary quadratic forms over } F} \right\} \\
	[H] & \mapsto \left[SL^\ddagger(2,H)\right].
	\end{align*}
\end{theorem}

\begin{proof}
First, we must check that the given map is well-defined. In particular, we need to check that the choice of orthogonal involution $\ddagger$ does not change the isomorphism class of the spin group. We use the fact that all orthogonal involutions are conjugate---that is, if $\ddagger_1, \ddagger_2$ are orthogonal involutions on $H$, then there exists a $u \in H$ such that for all $x \in H$,
	\begin{align*}
	\left(u x u^{-1}\right)^{\ddagger_1} = u x^{\ddagger_2} u^{-1}.
	\end{align*}
	
\noindent But that means that $SL^{\ddagger_1}(2,H)$ and $SL^{\ddagger_2}(2,H)$ are conjugate inside of $GL(2,H)$, i.e. there is an isomorphism
	\begin{align*}
	SL^{\ddagger_2}(2,H) &\rightarrow SL^{\ddagger_1}(2,H) \\
	\begin{pmatrix} a & b \\ c & d \end{pmatrix} &\mapsto \begin{pmatrix} u & 0 \\ 0 & u \end{pmatrix} \begin{pmatrix} a & b \\ c & d \end{pmatrix}\begin{pmatrix} u^{-1} & 0 \\ 0 & u^{-1} \end{pmatrix}.
	\end{align*}
	
\noindent Therefore, the given map is well-defined. Next, we check that it is surjective. Choose any indefinite, quinary quadratic form $q$ over $F$. Since it is indefinite, we can decompose it as $\langle 1, -1 \rangle \oplus \langle a,b,c \rangle$, for some $a,b,c \in F^\times$. In fact, since scaling the quadratic form does not change the spin group, we can assume that the quadratic form is $\langle 1, -1 \rangle \oplus \langle 1,b,c \rangle$. In that case, it is clear that the image of
	\begin{align*}
	H = \left(\frac{-b,-c}{F}\right)
	\end{align*}
	
\noindent will be the desired spin group. So, we are finally left with checking that the map is injective, which is to say that if $SL^{\ddagger_1}(2,H_1)$ is isomorphic to $SL^{\ddagger_2}(2,H_2)$, then $H_1 \cong H_2$. The Lie algebras of $SL^{\ddagger_1}(2,H_1)$ and $SL^{\ddagger_2}(2,H_2)$ are $\mathfrak{so}(q_1)$ and $\mathfrak{so}(q_2)$, respectively, where
	\begin{align*}
	q_i: F^2 \oplus H_i^+ &\rightarrow F \\
	(s,t,z) &\mapsto -st + \nrm(z).
	\end{align*}
	
\noindent However, if $SL^{\ddagger_1}(2,H_1) \cong SL^{\ddagger_2}(2,H_2)$, then this isomorphism induces an isomorphism $\mathfrak{so}(q_1) \cong \mathfrak{so}(q_2)$, which can only happen if $q_1 \cong \lambda q_2$ for some $\lambda \in F^\times$ by the classification of Lie algebras. By Witt cancellation, it follows that $q_1' \cong \lambda q_2'$, where
	\begin{align*}
	q_i': H_i^+ &\rightarrow F \\
	z &\mapsto \nrm(z).
	\end{align*}
	
\noindent From this, it follows that $q_1'' \cong \lambda' q_2''$ for some $\lambda' \in F$, where
	\begin{align*}
	q_i'': H_i^0 &\rightarrow F \\
	z &\mapsto \nrm(z).
	\end{align*}
	
\noindent Since both $q_1''$ and $q_2''$ have discriminant $1$, in fact it must be true that $q_1'' \cong q_2''$. However, it is well-known that this quadratic form determines the quaternion algebra---that is, $H_1 \cong H_2$.
\end{proof}

In the proof of Theorem \ref{Correspondence between QAs and Spin Groups}, we only used that $SL^{\ddagger_1}(2,H)$ and $SL^{\ddagger_2}(2,H)$ are conjugate inside of $GL(2,H)$. However, we can prove something a little stronger.

\begin{theorem}\label{Isomorphisms and Embeddings}
Let $F$ be a characteristic $0$ field, and let $H_1$, $H_2$ be isomorphic quaternion algebras over $F$, with orthogonal involutions $\ddagger_1$ and $\ddagger_2$. Then there exists a quaternion algebra $H$ over a field $F' \supset F$ with an orthogonal involution $\ddagger$ such that $(H_1, \ddagger_1)$ and $(H_2, \ddagger_2)$ embed inside of $(H, \ddagger)$. Furthermore, for any $F'$ and $(H,\ddagger)$ satisfying this condition, the natural embeddings of $SL^{\ddagger_1}(2,H_1)$ and $SL^{\ddagger_2}(2,H_2)$ are conjugate inside of $SL^\ddagger(2,H)$.
\end{theorem}

\begin{proof}
To see that such a quaternion algebra $H$ must exist, it suffices to take $\Mat(2,\overline{F})$ with transpose as the orthogonal involution, where $\overline{F}$ is the algebraic closure of $F$. On the other hand, given $(H,\ddagger)$ that $(H_1, \ddagger_1)$ and $(H_2, \ddagger_2)$ embed in, if there is an isomorphism between $H_1$ and $H_2$, that isomorphism can be extended to an automorphism of $H$, and therefore there must exist $u \in H$ such that this map is of the form $x \mapsto uxu^{-1}$. Since both $H_1$ and $H_2$ are closed under $\ddagger$, we have that
	\begin{align*}
	\left(uxu^{-1}\right)^\ddagger = \left(u^{-1}\right)^\ddagger x^\ddagger u^\ddagger \in u H_1 u^{-1},
	\end{align*}
	
\noindent hence
	\begin{align*}
	H_1 = u^\ddagger u H_1 \left(u^\ddagger u\right)^{-1},
	\end{align*}
	
\noindent which can only happen if $u^\ddagger u = v \in H_1^+ \cap H_1^\times$. However, proving conjugacy is now trivial---if
	\begin{align*}
	\begin{pmatrix} a & b \\ c & d \end{pmatrix} \in SL^{\ddagger_1}(2,H_1),
	\end{align*}
	
\noindent then
	\begin{align*}
	\begin{pmatrix} u & 0 \\ 0 & \left(u^{-1}\right)^\ddagger \end{pmatrix}\begin{pmatrix} a & b \\ c & d \end{pmatrix}\begin{pmatrix} u^{-1} & 0 \\ 0 & u^\ddagger \end{pmatrix} &= \begin{pmatrix} uau^{-1} & ubu^\ddagger \\ \left(u^{-1}\right)^\ddagger c u^{-1} & \left(u^{-1}\right)^\ddagger d u^\ddagger \end{pmatrix} \\ &= \begin{pmatrix} uxu^{-1} & u(bv)u^{-1} \\ u(v^{-1}c)u^{-1} & u(v^{-1}dv)u^{-1} \end{pmatrix} \in SL^{\ddagger_2}(2,H_2).
	\end{align*}
\end{proof}

It is clear that the isomorphism in constructed in the proof of Theorem \ref{Correspondence between QAs and Spin Groups} restricts to the identity on $SL(2,F)$. On the other hand, if we wish to only consider conjugacy inside of some larger group $SL^\ddagger(2,H)$, then this condition cannot be satisfied in general. This gives a connection between isomorphisms preserving the involution and conjugation preserving $SL(2,F)$.

\begin{theorem}\label{Algebraic Group Isomorphism Fixing Base Group}
Let $F$ be a characteristic $0$ field, and let $H_1$, $H_2$ be isomorphic quaternion algebras over $F$, with orthogonal involutions $\ddagger_1$ and $\ddagger_2$. The following two statements are equivalent.
	\begin{enumerate}
		\item If $H$ is a quaternion algebra over a field extension $F'$ with orthogonal involution $\ddagger$ such that $(H_1, \ddagger_1)$, $(H_2, \ddagger_2) \hookrightarrow (H,\ddagger)$, then there exists an element $\gamma \in SL^\ddagger(2,H)$ such that $\gamma SL^{\ddagger_1}(2,H_1) \gamma^{-1} = SL^{\ddagger_2}(2,H_2)$, and $\gamma SL(2,F) \gamma^{-1} = SL(2,F)$.
		\item $(H_1, \ddagger_1) \cong (H_2, \ddagger_2)$.
	\end{enumerate}
\end{theorem}

\begin{proof}
Suppose that the first condition holds, so $\gamma SL^{\ddagger_1}(2,H_1) \gamma^{-1} = SL^{\ddagger_2}(2,H_2)$, and $\gamma SL(2,F) \gamma^{-1} = SL(2,F)$. Note that $M \mapsto \gamma M \gamma^{-1}$ can be extended to an automorphism of $\Mat(2,F)$; since $\Mat(2,F)$ is a central simple algebra, we see that there exists $\gamma' \in GL(2,F)$ such that $M \mapsto \gamma'\gamma M (\gamma' \gamma)^{-1}$ acts as the identity on $SL(2,F)$. Define $\gamma'' = \gamma'\gamma$; the map $M \mapsto \gamma'' M {\gamma''}^{-1}$ must map the subgroups
	\begin{align*}
	\mathcal{D}_{H_i,\ddagger_i} = \left\{M \in SL^{\ddagger_i}(2,H_i)\middle|M\begin{pmatrix} \lambda & 0 \\ 0 & \lambda^{-1} \end{pmatrix} = \begin{pmatrix} \lambda & 0 \\ 0 & \lambda^{-1} \end{pmatrix}M, \ \forall \lambda \in F^\times\right\}
	\end{align*}
	
\noindent to each other. But, of course,
	\begin{align*}
	\begin{pmatrix} \lambda & 0 \\ 0 & \lambda^{-1} \end{pmatrix} \begin{pmatrix} a & b \\ c & d \end{pmatrix} &= \begin{pmatrix} \lambda a & \lambda b \\ \lambda^{-1} c & \lambda^{-1} d \end{pmatrix} \\
	\begin{pmatrix} a & b \\ c & d \end{pmatrix} \begin{pmatrix} \lambda & 0 \\ 0 & \lambda^{-1} \end{pmatrix} &= \begin{pmatrix} \lambda a & \lambda^{-1} b \\ \lambda c & \lambda^{-1} d \end{pmatrix},
	\end{align*}
	
\noindent hence $\mathcal{D}_{H_i, \ddagger_i}$ is just the subgroup of diagonal matrices. It follows that
	\begin{align*}
	\gamma'' = \begin{pmatrix} u & 0 \\ 0 & \lambda\left(u^{-1}\right)^\ddagger \end{pmatrix}
	\end{align*}
	
\noindent for some $u \in H^\times$, $\lambda \in F^\times$. Furthermore,
	\begin{align*}
	\begin{pmatrix} u & 0 \\ 0 & \lambda\left(u^{-1}\right)^\ddagger \end{pmatrix} \begin{pmatrix} 1 & 1 \\ 0 & 1 \end{pmatrix} \begin{pmatrix} u^{-1} & 0 \\ 0 & \lambda^{-1}u^\ddagger \end{pmatrix} = \begin{pmatrix} 1 & \frac{uu^\ddagger}{\lambda} \\ 0 & 1 \end{pmatrix} = \begin{pmatrix} 1 & 1 \\ 0 & 1 \end{pmatrix},
	\end{align*}
	
\noindent whence $uu^\ddagger = \lambda \in F^\times$, which implies that the map $z \mapsto uzu^{-1}$ is an isomorphism of algebras with involution. In the opposite direction, if $\phi: H_1 \rightarrow H_2$ is an isomorphism of algebras with involution, then $\phi \otimes 1: H_1 \otimes_F F' \rightarrow H_2 \otimes_F F'$ is an automorphism of $H$ preserving $\ddagger$; since $H$ is a central simple algebra, there must exist a $u \in H^\times$ such that $\phi(z) = uzu^{-1}$. Since $\phi$ preserves the involution, we know that $\left(uzu^{-1}\right)^\ddagger = (u^{-1})^\ddagger z^\ddagger u^\ddagger = uz^\ddagger u^{-1}$. This occurs if $u^\ddagger u = \lambda \in F^\times$. It is then easy to check that if $a,b,c,d \in F$, then
	\begin{align*}
	\underbrace{\begin{pmatrix} u & 0 \\ 0 & \left(u^{-1}\right)^\ddagger\end{pmatrix}}_{:=\gamma}\begin{pmatrix} a & b \\ c & d \end{pmatrix}\begin{pmatrix} u^{-1} & 0 \\ 0 & u^\ddagger\end{pmatrix} &= \begin{pmatrix} a & \lambda b \\ c/\lambda & d \end{pmatrix},
	\end{align*}
	
\noindent hence $\gamma$ satisfies the conditions of the theorem.
\end{proof}

\section{Maximal Arithmetic Subgroups:}\label{Maximal Subgroups Section}
Having described when algebraic groups constructed from orders with involution are isomorphic and/or conjugate, we turn our attention to arithmetic groups. Specifically, let $H$ be a quaternion algebra with involution $\ddagger$, over a number field $K$. If $\OO$ is a $\ddagger$-order, then $SL^\ddagger(2,\OO)$ is an arithmetic subgroup of $SL^\ddagger(2,H)$. In the special case where $K = \QQ$, $SL^\ddagger(2,\OO)$ is either a lattice in $SO^0(4,1)$ or in $SO^0(3,2)$---we claim that in either case, it is actually a maximal arithmetic group, in the sense that it is not strictly contained inside any other arithmetic subgroup of $SL^\ddagger(2,H)$. To prove this, we shall make heavy use of the $\mathfrak{o}_K$-algebra generated by the elements $\Gamma \subset SL^\ddagger(2,H)$, which we shall denote by $\mathfrak{o}_K[\Gamma]$. First, we note that this is actually a group invariant.

\begin{lemma}\label{Ring is Invariant}
Let $H_1, H_2$ be rational quaternion algebras with orthogonal involutions $\ddagger_1, \ddagger_2$ over $K$. Let $\Gamma_1, \Gamma_2$ be lattices of $SL^{\ddagger_1}(2,H_1)$, $SL^{\ddagger_2}(2,H_2)$ such that they are Galois-closed and their centers are $\{\pm I\}$, where $I$ is the identity matrix. If $\Gamma_1$ and $\Gamma_2$ are isomorphic as groups, then $\mathfrak{o}_K[\Gamma_1]$ and $\mathfrak{o}_K[\Gamma_2]$ are isomorphic as rings.
\end{lemma}

\begin{proof}
First, note that $\phi(-I) = -I$, since $-I$ is the unique non-identity element of the centers of $\Gamma_i$. Therefore, it induces an isomorphism $\overline{\phi}: \overline{\Gamma}_1 \rightarrow \overline{\Gamma}_2$ between the images of the $\Gamma_i$ inside $SL^{\ddagger_i}(2,H_i)/\{\pm I\}$. Note that
	\begin{align*}
	SL^{\ddagger_i}(2,H_i \otimes_\QQ \RR)/\{\pm I\} \cong \begin{cases} PSO^0(4,1) & \text{if } H_i \otimes_\QQ \RR \cong H_\RR \\ PSO^0(3,2) & \text{if } H_i \otimes_\QQ \RR \cong \Mat(2,\RR), \end{cases}
	\end{align*}

\noindent and so we can apply the Mostow rigidity theorem to conclude that both $\overline{\Gamma_1}$ and $\overline{\Gamma_2}$ can be viewed as being lattices of the same Lie group $G$, namely either $PSO^0(3,2)$ or $PSO^0(4,1)$. In either case, $G$ is simple, and therefore by Mostow rigidity $\overline{\Gamma_1}$ and $\overline{\Gamma_2}$ are conjugate in $G$. Letting $g \in G$ be such that $\overline{\Gamma_2} = g \overline{\Gamma_1}g^{-1}$, we have a well-defined ring homomorphism
	\begin{align*}
	\Phi: \mathfrak{o}_K[\Gamma_1] &\rightarrow \mathfrak{o}_K[\Gamma_2] \\
	M &\mapsto g M g^{-1}.
	\end{align*}
	
\noindent This map is clearly invertible, and so we are done.
\end{proof}

\begin{remark}
The use of the Mostow rigidity theorem in Lemma \ref{Ring is Invariant} makes clear why we restrict to the case where $H$ is a quaternion algebra over $\QQ$---over other number fields, the set of infinite places $\Omega_\infty$ has more than one element, and so rather than $SL^\ddagger(2,\OO)/\{\pm I\}$ injecting as a lattice into $PSO^0(4,1)$ or $PSO^0(3,2)$, it will instead inject into some Lie group
	\begin{align*}
	\bigtimes_{\nu \in \Omega_\infty} G_\nu,
	\end{align*}
	
\noindent where for every $\nu$, $G_\nu$ is either $PSO^0(4,1)$ or $PSO^0(3,2)$---in any case, this group is no longer simple, and so Mostow rigidity cannot be applied. This is not to say that the other results of this section are necessarily false over number fields $K \neq \QQ$, but we do not pursue this question here.
\end{remark}

Next, we prove that this $\mathfrak{o}_K$-algebra is especially nice for groups $SL^\ddagger(2,\OO)$.

\begin{lemma}\label{Group Ring}
Let $H$ be a quaternion algebra with orthogonal involution $\ddagger$ over a number field $K$. Let $\OO$ be a maximal $\ddagger$-order of $H$. Then $\mathfrak{o}_K\left[SL^\ddagger(2,\OO)\right] = \Mat(2,\OO)$.
\end{lemma}

\begin{proof}
Obviously, $\mathfrak{o}_K\left[SL^\ddagger(2,\OO)\right]$ is contained inside of $\Mat(2,\OO)$, so it shall suffice to show that every element of $\Mat(2,\OO)$ can be is contained in $\mathfrak{o}_K\left[SL^\ddagger(2,\OO)\right]$. Since $\mathfrak{o}_K \subset \OO$, we know that $\Mat(2,\mathfrak{o}_K) \subset \mathfrak{o}_K\left[SL^\ddagger(2,\OO)\right]$. Consider the projection
	\begin{align*}
	\Psi: SL^\ddagger(2,\OO) &\rightarrow \OO/\mathfrak{o}_K \\
	\begin{pmatrix} a & b \\ c & d \end{pmatrix} &\mapsto a + \mathfrak{o}_K.
	\end{align*}
	
\noindent It is easy to see that for every prime ideal $\mathfrak{p} \subset \mathfrak{o}_K$, the corresponding map
	\begin{align*}
	\Psi_\mathfrak{p}: SL^\ddagger(2,\OO_\mathfrak{p}) &\rightarrow \OO_\mathfrak{p}/\mathfrak{o}_{K_\mathfrak{p}} \\
	\begin{pmatrix} a & b \\ c & d \end{pmatrix} &\mapsto a + \mathfrak{o}_{K_\mathfrak{p}}
	\end{align*}
	
\noindent is surjective---this is because for any $z \in \OO_\mathfrak{p}$, there exists some $\lambda \in \mathfrak{o}_{K_\mathfrak{p}}$ such that $\lambda + z \in \OO_\mathfrak{p}^\times$, in which case we have
	\begin{align*}
	\underbrace{\begin{pmatrix} \lambda + z & 0 \\ 0 & \left(\lambda + z^\ddagger\right)^{-1} \end{pmatrix}}_{\in SL^\ddagger(2,\OO_\mathfrak{p})} \mapsto z + \mathfrak{o}_{K_\mathfrak{p}}.
	\end{align*}
	
\noindent However, by Corollary \ref{Strong Approximation}, we can apply strong approximation to $SL^\ddagger(2,\OO)$, from which it follows that
	\begin{align*}
	\Psi\left(SL^\ddagger(2,\OO)\right) = \bigcap_{\mathfrak{p}} \Psi_{\mathfrak{p}}\left(SL^\ddagger(2,\OO_\mathfrak{p})\right) = \OO/\mathfrak{o}_K.
	\end{align*}

\noindent Ergo, for any $z \in \OO$, there exists an element $\gamma \in SL^\ddagger(2,\OO)$ and a $\lambda \in \mathfrak{o}_K$ such that
		\begin{align*}
		\begin{pmatrix} 1 & 0 \\ 0 & 0 \end{pmatrix}\gamma \begin{pmatrix} 1 & 0 \\ 0 & 0 \end{pmatrix} = \begin{pmatrix} \lambda + z & 0 \\ 0 & 0 \end{pmatrix} \in \mathfrak{o}_K\left[SL^\ddagger(2,\OO)\right].
		\end{align*}
		
\noindent Since the group ring also contains $\Mat(2,\mathfrak{o})$, we can conclude that it actually contains all of $\Mat(2,\OO)$.
\end{proof}

Finally, we shall need to know that the ring generated by an arithmetic group is an order.

\begin{lemma}\label{Lattices beget lattices}
Let $H$ be a rational quaternion algebra with orthogonal involution $\ddagger$. Let $\Gamma$ be an arithmetic subgroup of $SL^\ddagger(2,H)$. Then $\ZZ\left[\Gamma\right]$ is an order of the central simple algebra $\Mat(2,H)$.
\end{lemma}

\begin{proof}
First, note that since $\Gamma \subset SL^\ddagger(2,H)$, $\ZZ\left[\Gamma\right] \subset \ZZ\left[SL^\ddagger(2,H)\right] = \Mat(2,H)$. Since $\Gamma$ is an arithmetic group, for some integer $l$, there is a morphism $\Psi:\Gamma \rightarrow SL(l,\ZZ)$ with a finite kernel. It is easy to see that $\ZZ\left[SL(l,\ZZ)\right] = \Mat(l,\ZZ)$ is a finitely-generated, Noetherian $\ZZ$-module. Therefore, the sub-module $\ZZ\left[\Psi(\Gamma)\right]$ is finitely-generated. This, in turn, means that $\ZZ\left[\Gamma\right]$ is finitely-generated as an $\ZZ$-module. Since it is a subring of the finite-dimensional algebra $\Mat(2,H)$, it is therefore an order.
\end{proof}

With these three lemmas out of the way, we can prove the maximality of the groups $SL^\ddagger(2,\OO)$ for rational quaternion algebras.

\begin{theorem}\label{Maximal Arithmetic Groups}
Let $H$ be a rational quaternion algebra with orthogonal involution $\ddagger$. Let $\OO$ be a maximal $\ddagger$-order of $H$. Then $SL^\ddagger(2,\OO)$ is a maximal arithmetic subgroup of $SL^\ddagger(2,H)$.
\end{theorem}

\begin{proof}
Suppose that $\Gamma$ is an arithmetic group containing $SL^\ddagger(2,\OO)$. By Lemma \ref{Lattices beget lattices}, we know that $\ZZ\left[\Gamma\right]$ is an order of $\Mat(2,H)$ which, by Lemma \ref{Group Ring} contains $\Mat(2,\OO)$. Choose any element $\gamma \in \Gamma$, and choose any one of its coordinates $x$. Since the group ring contains
	\begin{align*}
	\begin{pmatrix} 1 & 0 \\ 0 & 0 \end{pmatrix},\begin{pmatrix} 0 & 1 \\ 0 & 0 \end{pmatrix},\begin{pmatrix} 0 & 0 \\ 1 & 0 \end{pmatrix},\begin{pmatrix} 0 & 0 \\ 0 & 1 \end{pmatrix},
	\end{align*}
	
\noindent it is clear that $\ZZ\left[\Gamma\right]$ must contain $\Mat(2,\OO[x])$. However, since
	\begin{align*}
	\begin{pmatrix} a & b \\ c & d \end{pmatrix}^{-1} = \begin{pmatrix} d^\ddagger & -b^\ddagger \\ -c^\ddagger & a^\ddagger \end{pmatrix},
	\end{align*}
	
\noindent we see that it must actually contain $\Mat(2,\OO[x,x^\ddagger])$. Since $\Mat(2,\OO[x,x^\ddagger])$ is a subring of the group ring, it must also be an order, from which we get that $\OO[x,x^\ddagger]$ is an order. However, $\OO[x,x^\ddagger]$ is clearly closed under $\ddagger$, hence it is a $\ddagger$-order. Since $\OO$ is a maximal $\ddagger$-order, it follows that $\OO[x,x^\ddagger] = \OO$. Therefore, $\Gamma \subset \Mat(2,\OO)$. However, $SL^\ddagger(2,\OO) = \Mat(2,\OO) \cap SL^\ddagger(2,H)$, therefore $\Gamma = SL^\ddagger(2,\OO)$.
\end{proof}

\section{Isomorphisms Between Groups $SL^\ddagger(2,\OO)$:}\label{Isomorphisms Between Arithmetic Groups}
As in Section \ref{Maximal Subgroups Section}, we consider quaternion algebras $H$ over $\QQ$. For any choice of orthogonal involutions $\ddagger_1, \ddagger_2$ and corresponding maximal $\ddagger$-orders $\OO_1, \OO_2$ of $H$, $SL^{\ddagger_1}(2,\OO_1)$ will be commensurable with $SL^{\ddagger_2}(2,\OO_2)$, since by Theorem \ref{Correspondence between QAs and Spin Groups} they are arithmetic groups of a spin group. We shall want to study when such groups are isomorphic, and furthermore when we can say that they are conjugate inside of some group. We begin by establishing a necessary condition.

\begin{theorem}\label{Isomorphism Necessary Condition}
Let $H_1,H_2$ be rational quaternion algebras, with orthogonal involutions $\ddagger_1$ and $\ddagger_2$. Let $\OO_1,\OO_2$ be maximal $\ddagger_1,\ddagger_2$-orders of $H_1$ and $H_2$, respectively. If the group $SL^{\ddagger_1}(2,\OO_1)$ is isomorphic to $SL^{\ddagger_2}(2,\OO_2)$, then $H_1 \cong H_2$ and $\Mat(2,\OO_1) \cong \Mat(2, \OO_2)$. In particular, $\disc(\OO_1) = \disc(\OO_2)$.
\end{theorem}

\begin{proof}
By Lemmas \ref{Ring is Invariant} and \ref{Group Ring}, if the groups are isomorphic, then $\Mat(2,\OO_1) \cong \Mat(2,\OO_2)$ as rings. However, if $\Mat(2,\OO_1)$ and $\Mat(2,\OO_2)$ are isomorphic, then they have the same discriminant and it is a standard exercise to check that $\disc\left(\Mat(2,\OO_i)\right) = \disc(\OO_i)$.
\end{proof}

\begin{remark}
A tempting generalization of Theorem \ref{Isomorphism Necessary Condition} would be that if $SL^{\ddagger_1}(2,\OO_1) \cong SL^{\ddagger_2}(2,\OO_2)$, then $\OO_1 \cong \OO_2$. However, this statement is false.
\end{remark}

\begin{example}\label{Non Isomorphic Orders with Isomorphic Groups}
Take $\OO_1, \OO_2$ as in Example \ref{Non Isomorphic Orders}. We proved that $\OO_1 \ncong \OO_2$; however, we claim that $SL^\ddagger(2,\OO_1) \cong SL^\ddagger(2,\OO_2)$.
\end{example}

\begin{proof}
Define
	\begin{align*}
	\gamma = \begin{pmatrix} \frac{1 - 6i + j}{2\sqrt{3}} & \frac{1 + j}{2\sqrt{3}} \\ \frac{1 + 6i + j}{2\sqrt{3}} & \sqrt{3} i \end{pmatrix} \in SL^\ddagger\left(2,H \otimes_\QQ \QQ(\sqrt{3})\right).
	\end{align*}
	
\noindent By an easy computation,

	\begin{minipage}{.45\textwidth}
	\begin{align*}
	\gamma \begin{pmatrix} 0 & 1 \\ 0 & 0 \end{pmatrix} \gamma^{-1} &\in \Mat(2,\OO_2) \\
	\gamma \begin{pmatrix} 0 & i \\ 0 & 0 \end{pmatrix} \gamma^{-1} &\in \Mat(2,\OO_2) \\
	\gamma \begin{pmatrix} 0 & \frac{1 + j}{2} \\ 0 & 0 \end{pmatrix} \gamma^{-1} &\in \Mat(2,\OO_2)
	\end{align*}
	\end{minipage}%
	\begin{minipage}{.45\textwidth}
	\begin{align*}
	\gamma \begin{pmatrix} 0 & \frac{i + ij}{2} \\ 0 & 0 \end{pmatrix} \gamma^{-1} &\in \Mat(2,\OO_2) \\
	\gamma \begin{pmatrix} 0 & 0 \\ 1 & 0 \end{pmatrix} \gamma^{-1} &\in \Mat(2,\OO_2),
	\end{align*}
	\end{minipage}
	
\noindent from which it follows that there is a well-defined, injective ring homomorphism
	\begin{align*}
	\Mat(2,\OO_1) &\rightarrow \Mat(2,\OO_2) \\
	M &\mapsto \gamma M \gamma^{-1},
	\end{align*}
	
\noindent since the conjugated elements generate $\Mat(2,\OO_1)$ as a ring over $\ZZ$. As $\gamma \in SL^\ddagger\left(2,H \otimes_\QQ \QQ(\sqrt{3})\right)$, this homomorphism restricts to an injective group homomorphism
	\begin{align*}
	SL^\ddagger(2,\OO_1) &\rightarrow SL^\ddagger(2,\OO_2) \\
	M &\mapsto \gamma M \gamma^{-1}.
	\end{align*}
	
\noindent By the maximality of $SL^\ddagger(2,\OO_1)$ and $SL^\ddagger(2,\OO_2)$, this must be an isomorphism.
\end{proof}

\begin{remark}
A corollary of Example \ref{Non Isomorphic Orders with Isomorphic Groups} and Theorem \ref{Isomorphism Necessary Condition} is that $\OO_1, \OO_2$ are non-isomorphic rings such that $\Mat(2,\OO_1) \cong \Mat(2,\OO_2)$. While this does produce a new proof of that fact, this example was in fact already worked out by Chatters \cite[Example 5.1]{Chatters1996}
\end{remark}

We shall also seek an analog of the notion of isomorphism from Theorem \ref{Algebraic Group Isomorphism Fixing Base Group}---specifically, we want a notion of isomorphism that corresponds to the notion of isomorphism of algebras with involution. This has the benefit that it can be stated over any number field.

\begin{theorem}\label{Special Isomorphism of Orders Is Equivalent to Special Isomorphism of Groups}
Let $H$ a quaternion algebra over a number field $K$ with orthogonal involution $\ddagger$. Let $\OO_1, \OO_2$ be maximal $\ddagger$-orders of $H$. Let $K'$ be a field extension of $K$ containing the splitting fields of $X^2 = \nrm(x)$ for all $x \in H$. Then the following are equivalent.
	\begin{enumerate}
	\item There exists $\gamma \in SL^\ddagger(2,H \otimes_K K')$ such that $\gamma SL^\ddagger(2,\OO_1) \gamma^{-1} = SL^\ddagger(2,\OO_2)$ and $\gamma SL(2,K) \gamma^{-1} = SL(2,K)$.
	\item $(\OO_1,\ddagger) \cong (\OO_2,\ddagger)$.
	\end{enumerate}
\end{theorem}

\begin{proof}
The proof is similar to that of Theorem \ref{Algebraic Group Isomorphism Fixing Base Group}. If $\OO_1 \cong_\ddagger \OO_2$, then there exists $x \in H^\times$ such that $x^\ddagger = \overline{x}$ such that $\OO_2 = x \OO_1 x^{-1}$. Then $x' = x \otimes 1/\sqrt{\nrm(x)} \in H \otimes_F F'$ is an element of norm $\pm 1$, and so if we take
	\begin{align*}
	\gamma = \begin{pmatrix} x & 0 \\ 0 & \left(x^{-1}\right)^\ddagger \end{pmatrix} \in SL^\ddagger(2, H \otimes_F F'),
	\end{align*}
	
\noindent it will have the desired effect. In the other direction, if there exists $\gamma \in SL^\ddagger(2,H \otimes_F F')$ such that $\gamma SL^\ddagger(2,\OO_1) \gamma^{-1} = SL^\ddagger(2,\OO_2)$ and $\gamma SL(2,F) \gamma^{-1} = SL(2,F)$, then there exists $\gamma' \in SL^\ddagger(2,H \otimes \overline{F})$ such that $\gamma' SL^\ddagger(2,\OO_1) {\gamma'}^{-1} = SL^\ddagger(2,\OO_2)$ and $\gamma' M {\gamma'}^{-1} = M$ for all $M \in SL(2,F)$. This induces a ring isomorphism on the corresponding group rings; that is, by Lemma \ref{Group Ring}, there is a ring isomorphism
	\begin{align*}
	\Mat(2,\OO_1) &\rightarrow \Mat(2,\OO_2) \\
	M &\mapsto \gamma' M {\gamma'}^{-1}
	\end{align*}
	
\noindent such that this isomorphism restricts to the identity on $\Mat(2,\mathfrak{o}_K)$. In particular, we can define subrings
	\begin{align*}
	U_i = \left\{M \in \Mat(2,\OO_i) \middle| M\begin{pmatrix} 1 & 1 \\ 0 & 1 \end{pmatrix} = \begin{pmatrix} 1 & 1 \\ 0 & 1 \end{pmatrix}M\right\},
	\end{align*}
	
\noindent and we are guaranteed that the ring isomorphism between the $\Mat(2,\OO_i)$ restricts to a ring isomorphism between the $U_i$. However, it is easy to see that $M \in U_i$ if and only if it is upper triangular. Therefore, for any $z \in \OO_1$,
	\begin{align*}
	\gamma' \begin{pmatrix} 1 & z \\ 0 & 1 \end{pmatrix} {\gamma'}^{-1} \in U_2.
	\end{align*}
	
\noindent Write
	\begin{align*}
	\gamma' = \begin{pmatrix} a & b \\ c & d \end{pmatrix} \in SL^\ddagger(2,H \otimes_F \overline{F}),
	\end{align*}
	
\noindent and note that
	\begin{align*}
	\gamma' \begin{pmatrix} 1 & z \\ 0 & 1 \end{pmatrix} {\gamma'}^{-1} = \begin{pmatrix} * & * \\ czc^\ddagger & * \end{pmatrix},
	\end{align*}
	
\noindent which belongs to $U_2$ if and only if $czc^\ddagger = 0$. However, since $czc^\ddagger = 0$ for all $z \in \OO_1$, it must be that $czc^\ddagger = 0$ for all $z \in H$. It is readily checked that this is possible if and only if $c = 0$---if $H$ is a division algebra, this assertion is immediate, and otherwise $H \cong \Mat(2,\overline{F})$ where it becomes an easy calculation. By a similar argument with the sub-ring
	\begin{align*}
	L_i = \left\{M \in \Mat(2,\OO_i) \middle| M\begin{pmatrix} 1 & 0 \\ 1 & 1 \end{pmatrix} = \begin{pmatrix} 1 & 0 \\ 1 & 1 \end{pmatrix}M\right\},
	\end{align*}
	
\noindent we can prove that $b = 0$ as well. Thus $\gamma'$ is a diagonal matrix, which is to say that
	\begin{align*}
	\gamma' = \begin{pmatrix} u & 0 \\ 0 & \left(u^{-1}\right)^\ddagger \end{pmatrix}
	\end{align*}
	
\noindent for some $u \in H^\times$. Thus $\OO_2 = u\OO_1 u^{-1}$. Since
	\begin{align*}
	\begin{pmatrix} u & 0 \\ 0 & \left(u^{-1}\right)^\ddagger \end{pmatrix} \begin{pmatrix} 1 & 1 \\ 0 & 1 \end{pmatrix} \begin{pmatrix} u^{-1} & 0 \\ 0 & u^\ddagger \end{pmatrix} &= \begin{pmatrix} 1 & uu^\ddagger \\ 0 & 1 \end{pmatrix} = \begin{pmatrix} 1 & 1 \\ 0 & 1 \end{pmatrix},
	\end{align*}
	
\noindent and so in fact $(\OO_1,\ddagger) \cong (\OO_2,\ddagger)$.
\end{proof}

\begin{remark}
Note that in the statement of Theorem \ref{Algebraic Group Isomorphism Fixing Base Group}, we only had to select the field extension $K'$ so that we could embed the quaternion algebras with involution into a single quaternion algebra with involution. If Theorem \ref{Special Isomorphism of Orders Is Equivalent to Special Isomorphism of Groups} were entirely analogous, we would be able to take $K' = K$. We shall show in Section \ref{Conjugacy Section} that this strengthening of the statement is false: there exist maximal $\ddagger$-orders $\OO_1, \OO_2$ that are isomorphic as algebras with involution, but such that $SL^\ddagger(2,\OO_1)$ and $SL^\ddagger(2,\OO_2)$ are not conjugate inside $SL^\ddagger(2,H)$.
\end{remark}

\begin{remark}
In the statement of Theorem \ref{Special Isomorphism of Orders Is Equivalent to Special Isomorphism of Groups}, the requirement that we choose $\gamma \in SL^\ddagger(2,H \otimes_K K')$ such that $\gamma SL(2,F) \gamma^{-1} = SL(2,F)$ is necessary. Specifically, there exist maximal $\ddagger$-orders $\OO_1, \OO_2$ such that $(\OO_1, \ddagger) \ncong (\OO_2, \ddagger)$ but they are conjugate inside of $SL^\ddagger(2,H\otimes_K K')$ for some field extension $K'$. An example was produced in \cite{Sheydvasser2019}, but as it was stated in somewhat different language, we reproduce it here.
\end{remark}

\begin{example}\label{Isomorphic but not Conjugate Groups}
Let
	\begin{align*}
	\OO_1 &= \ZZ \oplus \ZZ i \oplus \ZZ \frac{1 + j}{2} \oplus \ZZ \frac{i + ij}{2} \subset \left(\frac{-1,-7}{\QQ}\right) \\
	\OO_2 &= \ZZ \oplus \ZZ i \oplus \ZZ \frac{i + j}{2} \oplus \ZZ \frac{1 + ij}{2} \subset \left(\frac{-1,-7}{\QQ}\right).
	\end{align*}
	
\noindent Both of these are maximal $\ddagger$-orders, if we take the usual involution $\left(w + xi + yj + zij\right)^\ddagger = w + xi + yj - zij$. Since $\tr(\OO_1 \cap H^+) = (1)$ and $\tr(\OO_2 \cap H^+) = (2)$, we see that $(\OO_1,\ddagger) \ncong (\OO_2,\ddagger)$. However, $(1 + i) \OO_1 (1 + i)^{-1} = \OO_2$, and therefore
	\begin{align*}
	\begin{pmatrix} \frac{1 + i}{\sqrt{2}} & 0 \\ 0 & \frac{-1 + i}{\sqrt{2}} \end{pmatrix} SL^\ddagger(2,\OO_1) \begin{pmatrix} \frac{1 + i}{\sqrt{2}} & 0 \\ 0 & \frac{-1 + i}{\sqrt{2}} \end{pmatrix}^{-1} &= SL^\ddagger(2,\OO_2).
	\end{align*}
\end{example}

\section{Conjugacy Classes in $SL^\ddagger(2,H)$:}\label{Conjugacy Section}
Our results thus far have shown cases where you can prove that if $SL^\ddagger(2,\OO_1)$ and $SL^\ddagger(2,\OO_2)$ are isomorphic as groups, then they are conjugate inside of $SL^\ddagger(2,H\otimes_K K')$ where $K'$ is an extension of $K$. Examples \ref{Non Isomorphic Orders with Isomorphic Groups} and \ref{Isomorphic but not Conjugate Groups} demonstrate that $K'$ might be a quadratic extension of $K$. Our goal here is to give some insight as to when we can take $K' = K$. Our main result is that determining conjugacy is actually a local problem.

\begin{theorem}\label{Local Rigidity}
Let $H$ be a quaternion algebra over a number field $K$, with orthogonal involution	$\ddagger$. Let $\OO_1, \OO_2$ be maximal $\ddagger$-orders of $H$. Then $SL^\ddagger(2,\OO_1)$ is conjugate to $SL^\ddagger(2,\OO_2)$ in $SL^\ddagger(2,H)$ if and only if for every prime ideal $\mathfrak{p}$ of $\mathfrak{o}_K$, $SL^\ddagger(2,\OO_{1,\mathfrak{p}})$ is conjugate to $SL^\ddagger(2,\OO_{2,\mathfrak{p}})$ in $SL^\ddagger(2,H_\mathfrak{p})$.
\end{theorem}

\begin{proof}
Obviously, if $SL^\ddagger(2,\OO_1)$, $SL^\ddagger(2,\OO_2)$ are conjugate, then all of their localizations will be conjugate. In the other direction, for almost all prime ideals $\mathfrak{p}$, $\OO_{1,\mathfrak{p}} = \OO_{2,\mathfrak{p}}$---therefore, we can choose elements $\gamma_{\mathfrak{p}} \in SL^\ddagger(2,H_\mathfrak{p})$ such that $\gamma_{\mathfrak{p}}SL^\ddagger(2,\OO_{1,\mathfrak{p}})\gamma_{\mathfrak{p}}^{-1} = SL^\ddagger(2,\OO_{2,\mathfrak{p}})$ and $\gamma_{\mathfrak{p}} = id$ for almost all prime ideals $\mathfrak{p}$. Therefore, there exists an element $\gamma' \in SL^\ddagger(2,H \otimes_K \mathbb{A}_K)$ of the adelic extension such that the $\mathfrak{p}$-th component of $\gamma'$ is the $\gamma_{\mathfrak{p}}$. If $\gamma' \in SL^\ddagger(2,H \otimes_K \mathbb{A}_K)$ is such that $\left(\gamma'\right)_{\mathfrak{p}}SL^\ddagger(2,\OO_{1,\mathfrak{p}})\left(\gamma'\right)_{\mathfrak{p}}^{-1} = SL^\ddagger(2,\OO_{2,\mathfrak{p}})$ for all prime ideals $\mathfrak{p}$, then the same must true of every element of $SL^\ddagger(2,\OO_2) \gamma'$. Therefore, such elements form an open set in $SL^\ddagger(2,H \otimes_K \mathbb{A}_K)$---in that case, Corollary \ref{Strong Approximation} proves that there must exist a $\gamma \in SL^\ddagger(2,H)$ satisfying this condition.
\end{proof}

Thus, determining whether groups $SL^\ddagger(2,\OO)$ are conjugate in $SL^\ddagger(2,H)$ can be reduced to checking the same over local fields, where it is far easier. It is worth noting that even if $(\OO_1, \ddagger) \cong (\OO_2,\ddagger)$, it does not follow that $SL^\ddagger(2,\OO_1)$ is conjugate to $SL^\ddagger(2,\OO_2)$ inside $SL^\ddagger(2,H)$; one may need to pass to a higher degree extension.

\begin{example}
Let $F$ be any characteristic $0$ local field with maximal ideal $\mathfrak{p}$. Choose any $\lambda \in \mathfrak{o}_F$ such that $1 + \lambda \in \mathfrak{p} \backslash \mathfrak{p}^2$. Then $H = \Mat(2,F)$ is a quaternion algebra over $F$, and
	\begin{align*}
	\begin{pmatrix} a & b \\ c & d \end{pmatrix}^\ddagger = \begin{pmatrix} a & c/\lambda \\ b\lambda & d \end{pmatrix}
	\end{align*}
	
\noindent defines an orthogonal involution. Since $\lambda \in \mathfrak{o}^\times$, $\OO_1 = \Mat(2,\mathfrak{o}_F)$ is a maximal $\ddagger$-order. Similarly,
	\begin{align*}
	\OO_2 = \underbrace{\begin{pmatrix} 1 & -1 \\ \lambda & 1 \end{pmatrix}}_{:= u} \OO_1 \begin{pmatrix} 1 & -1 \\ \lambda & 1 \end{pmatrix}^{-1}
	\end{align*}
	
\noindent must be a maximal $\ddagger$-order; indeed, $(\OO_1,\ddagger) \cong (\OO_2,\ddagger)$. However, $SL^\ddagger(2,\OO_1)$ and $SL^\ddagger(2,\OO_2)$ are not conjugate in $SL^\ddagger(2,H)$.
\end{example}

\begin{proof}
Suppose that there exist $a,b,c,d \in H$ such that
	\begin{align*}
	\gamma := \begin{pmatrix} a & b \\ c & d \end{pmatrix} \in SL^\ddagger(2,H)
	\end{align*}
	
\noindent and $\gamma SL^\ddagger(2,\OO_1)\gamma^{-1} = SL^\ddagger(2,\OO_2)$. Since
	\begin{align*}
	\begin{pmatrix} a & b \\ c & d \end{pmatrix}\begin{pmatrix} 1 & z \\ 0 & 1 \end{pmatrix}\begin{pmatrix} a & b \\ c & d \end{pmatrix}^{-1} &= \begin{pmatrix} * & aza^\ddagger \\ -czc^\ddagger & * \end{pmatrix} \\
	\begin{pmatrix} a & b \\ c & d \end{pmatrix}\begin{pmatrix} 1 & 0 \\ z & 1 \end{pmatrix}\begin{pmatrix} a & b \\ c & d \end{pmatrix}^{-1} &= \begin{pmatrix} * & -bzb^\ddagger \\ dzd^\ddagger & * \end{pmatrix},
	\end{align*}
	
\noindent we wish to determine for which $v \in H$ $v\OO_1^+ v^\ddagger \subset \OO_2$. This is the same as determining all $v \in H$ such that $u^{-1}v\OO_1^+ v^\ddagger u \subset \OO_1$. We shall show that this is possible only if $v \in \OO_1$, proving that $\gamma \in SL^\ddagger(2,\OO_1)$. Write
	\begin{align*}
	v = \begin{pmatrix} v_1 & v_2 \\ v_3 & v_4 \end{pmatrix},
	\end{align*}
	
\noindent so
	\begin{align*}
	u^{-1}v \begin{pmatrix} 1 & 0 \\ 0 & 0 \end{pmatrix} v^\ddagger u &= \begin{pmatrix}
 \frac{\left(v_1+v_3\right)^2}{\lambda +1} & -\frac{\left(\lambda  v_1-v_3\right) \left(v_1+v_3\right)}{\lambda (\lambda +1)} \\
 -\frac{\left(\lambda  v_1-v_3\right) \left(v_1+v_3\right)}{\lambda +1} & \frac{\left(\lambda  v_1-v_3\right)^2}{\lambda  (\lambda +1)} \end{pmatrix} \in \Mat(2,\mathfrak{o}_F) \\
	u^{-1}v \begin{pmatrix} 0 & 0 \\ 0 & 1 \end{pmatrix} v^\ddagger u &= \begin{pmatrix} \frac{\lambda  \left(v_2+v_4\right)^2}{\lambda +1} & -\frac{\left(\lambda  v_2-v_4\right) \left(v_2+v_4\right)}{\lambda +1} \\ -\frac{\lambda  \left(\lambda  v_2-v_4\right) \left(v_2+v_4\right)}{\lambda +1} & \frac{\left(\lambda  v_2-v_4\right)^2}{\lambda +1} \end{pmatrix} \in \Mat(2,\mathfrak{o}_F).
	\end{align*}
	
\noindent Note that this is only possible if $v_1 + v_3, \lambda v_1 - v_3, v_2 + v_4, \lambda v_2 - v_4 \in \mathfrak{p}$. From it follows that $(1 + \lambda)v_1,(1 + \lambda)v_3 \in \mathfrak{p}$, hence $v_1, v_3 \in \mathfrak{o}_F$. Therefore, $v_3, v_4 \in \mathfrak{o}_F$. We conclude that $v \in \OO_1$. Ergo, $\gamma \in SL^\ddagger(2,\OO_1)$, and so $\gamma SL^\ddagger(2,\OO_1) \gamma^{-1} = SL^\ddagger(2,\OO_1) = SL^\ddagger(2,\OO_2)$. However, since
	\begin{align*}
	u \begin{pmatrix} 1 & 0 \\ 0 & 0 \end{pmatrix} u^{-1} = \begin{pmatrix} \frac{1}{1+\lambda } & \frac{1}{1 + \lambda} \\ \frac{\lambda }{1+\lambda } & \frac{\lambda }{1 + \lambda} \end{pmatrix} \notin \OO_1,
	\end{align*}
	
\noindent $\OO_1 \neq \OO_2$, and so $SL^\ddagger(2,\OO_1) \neq SL^\ddagger(2,\OO_2)$. This is a contradiction, and so we are done.
\end{proof}

\bibliography{AlgebraicIsomorphism}
\bibliographystyle{alpha}
\end{document}